\documentclass[12pt]{article}
\usepackage{amsmath, amsthm, amssymb}

\newtheorem{theorem}{Theorem}
\newtheorem{remark}{Remark}
\newtheorem{corollary}{Corollary}
\newtheorem{lemma}{Lemma}
\newtheorem{example}{Example}

\newtheorem{assumption}{Assumption}

\newcommand{\bard}{{\bar\delta}}

\numberwithin{equation}{section}

\begin{document}
 
 \title{Too much regularity may force too much uniqueness}

\author{
Martin Stynes\thanks{\texttt{m.stynes@csrc.ac.cn} The research of this author is supported in part by the National Natural Science Foundation of China under grant 91430216.}\\[2pt]
Applied and Computational Mathematics Division,\\ Beijing Computational Science Research Center,\\ 
Haidian District, Beijing 100193, China.
}

\maketitle

 \begin{abstract}

Time-dependent fractional-derivative problems $D_t^\delta u + Au = f$ are considered, where $D_t^\delta$ is a Caputo fractional derivative of order $\delta\in (0,1)\cup (1,2)$ and~$A$ is a classical elliptic operator, and appropriate boundary and initial conditions are applied. The regularity of solutions to this class of problems is discussed, and it is shown that assuming more regularity than is generally true---as many researchers do---places a surprisingly severe restriction on the problem.

 \medskip


 \smallskip


 \end{abstract}

 \section{The problem and the singularity in its solution} \label{sec:probsing}
 
 \subsection{Introduction and problem statement}\label{sec:intro}
Time-dependent problems with a fractional  temporal derivative of order $\delta\in (0,1)\cup (1,2)$ have attracted much attention in recent years, since these problems arise in several models of physical processes---see the references in~\cite{JLPZ15,LiXU10,McL10,MK00,SY11}. In the present paper we shall discuss the regularity of solutions in this class and a surprising consequence of assuming more regularity of the solution than is true in general.

Let $\Omega$ be a bounded domain in $\mathbb{R}^n$ for some $n\ge 1$. Let $\bar \Omega$ and $\partial\Omega$ denote its closure and boundary. Let $\delta\in (0,1)\cup (1,2)$. Set
\[
\bard =\begin{cases}
	1 &\text{if } 0<\delta<1, \\
	2 &\text{if } 1<\delta<2.
	\end{cases}
\]
Given a suitable function $g(x,t)$ defined on $\Omega\times [0,T]$ for some $T>0$, define \cite{Diet10} the \emph{Caputo fractional derivative} $D^\delta_t$  by   
$$
D_t^\delta g(x,t) := \frac1{\Gamma(\bard-\delta)}\int_{s=0}^t (t-s)^{\bard-\delta-1}\,\frac{\partial^\bard g(x,s)}{\partial s^\bard}\,ds  \quad\text{for } x\in\Omega, \, 0< t \le T.
$$

In this paper we shall consider the initial-boundary value problem 
\begin{subequations} \label{prob}
\begin{align}
Lu:= &D_t^\delta u-\sum_{i,j=1}^n p_{ij}(x,t)  \frac{\partial^2 u}{\partial x_i\partial x_j} + \sum_{i=1}^n q_i(x,t)  \frac{\partial u}{\partial x_i} + r(x,t) u=f(x,t)\label{proba}
\intertext{for $(x,t)\in Q:=\Omega\times(0,T]$, with}
u(x,t)&= \psi(x,t)\quad\text{for }(x,t) \in \partial\Omega \times (0,T], \label{probb} \\
u(x,0)&=\phi_0(x) \quad\text{for } x\in\bar\Omega, \label{probc} \\
u_t(x,0) &= \phi_1(x) \quad\text{for } x\in\Omega, \label{probd}
\end{align}
\end{subequations}
where the initial condition~\eqref{probd} is applied only when $1< \delta <2$. We assume that  the operator 
\[
w\mapsto \sum_{i,j=1}^n p_{ij}(x,t)  \partial^2 w/\partial x_i\partial x_j + \sum_{i=1}^n q_i(x,t)  \partial w/\partial x_i + r(x,t) w
\]
 is uniformly elliptic on~$Q$,  and that the functions $p_{ij},q_i,r,\psi, \phi_0$ and $\phi_1$ are continuous on the closures of their domains. We also assume the minimal amount of compatibility between the initial and boundary conditions so that the solution~$u$ is continuous on $\bar Q:= \bar\Omega \times [0,T]$: 
\begin{equation}\label{compatibility}
\phi_0(x) = \psi(x,0)  \ \text{ for all }x\in\partial\Omega.
\end{equation}


When $0<\delta <1$, the problem \eqref{prob} is a fractional-derivative generalisation of classical parabolic problems such as the heat equation, while for $1<\delta <2$ it generalises classical second-order hyperbolic problems such as the wave equation. Problems of the form \eqref{prob} have been considered in a huge number of papers in the research literature.

Under suitable smoothness and compatibility hypotheses on the data of \eqref{prob}, one can show existence and uniqueness of solutions to this problem in various spaces; see, e.g., \cite{CLP06,JLZ14,LiXU10,Luchko12,McL10,SY11}. It is not necessary to go into the details of these results here, except to note that they show that in general the solution~$u$ of \eqref{prob} is \emph{not} smooth in the closed domain~$\bar Q$, even when the data of the problem are smooth:
 when $0< \delta < 1$, in general~$u_t$ blows up at $t=0$, and when $1< \delta < 2$, in general  $u_{tt}$  blows up at $t=0$.  Example~\ref{exa:1} below exhibits the singularity at $t=0$ that is typical of solutions of~\eqref{prob}.
 
 Throughout the paper we are interested only in classical solutions of \eqref{prob}, i.e., functions $u$ whose derivatives $D_t^\delta u, \partial u/\partial x_i$ and $\partial^2 u/\partial x_i\partial x_j$ exist at all points in~$Q$ and satisfy \eqref{proba}  and its initial and boundary conditions pointwise.

\emph{Notation.}
Let $C^k(S)$, for any domain $S$, denote the space of real-valued functions whose derivatives of order $0,1,\dots, k$ are continuous on $S$. When $k=0$ we write $C(S)$.

\subsection{A typical example}
The Mittag-Leffler function $E_\alpha$ is defined \cite[Chapter 4]{Diet10} by
\[
E_{\alpha} (z) := \sum _{k=0}^\infty \frac{z^k}{\Gamma (\alpha k +1)}\,. 
\]
This series converges uniformly and absolutely for $\alpha >0$.
Using \cite[Appendix D]{Diet10} to differentiate the series term by term, which is easily justified, one sees that 
\begin{equation}\label{DdE}
D_t^\delta E_{\delta} (t^\delta) = E_{\delta} (t^\delta)\ \text{ for }t>0.
\end{equation}
Likewise, a simple differentiation yields
\begin{equation}\label{D1E}
\frac{d}{dt} E_{\delta} (t^\delta) = \sum_{k=1}^\infty \frac{t^{k\delta-1}}{\Gamma (\delta k)}\ \text{ for }t>0.
\end{equation}

\begin{example} \label{exa:1}
Consider the fractional heat/wave equation   
\[
D_t^\delta v - \partial^2 v/\partial x^2 = 0\ \text{ for } (x,t)\in (0,\pi)\times(0,T]
\]
with boundary conditions $v(0,t) = v(\pi,t) =0$ and the initial condition $v(x,0) = \sin x$ when $0<\delta<1$, and when $1<\delta <2$ the additional initial condition $v_t(x,0) = 0$.  From~\eqref{DdE} (and~\eqref{D1E} when $1<\delta<2$) it follows that the solution of this initial-boundary value problem is 
\[v(x,t)=E_{\delta}(-t^\delta)\sin x \ \text{ for } (x,t)\in[0,\pi]\times[0,T].
\] 
Hence 
\[
v_t(x,t) =  \sum_{k=1}^\infty \frac{(-1)^kt^{\delta k-1}}{\Gamma (\delta k)} \sin x\quad \text{and}\quad 
	v_{tt}(x,t) =  \sum_{k=1}^\infty \frac{(-1)^kt^{\delta k-2}}{\Gamma (\delta k -1)} \sin x
\]
for $(x,t)\in [0,\pi]\times(0,T]$.
If $x\in (0,\pi)$, it follows that for $0<\delta<1$, $v_t(x,t) \sim t^{\delta-1}$ as $t\to 0^+$, while for $1<\delta<2$ one has $v_{tt}(x,t) \sim t^{\delta-2}$ as $t\to 0^+$. These singularities in the temporal derivatives at $t=0$ are typical of solutions to the general problem~\eqref{prob}.
\end{example}

Note that all the data of Example~\ref{exa:1} are smooth; the cause of the singularity is the fractional derivative in the differential operator.

When the data of \eqref{prob} are smooth, one expects the pure spatial derivatives of the solution of~\eqref{prob} to be smooth globally, as can be seen in Example~\ref{exa:1}; only the temporal derivatives exhibit singularities at $t=0$.

\begin{remark}
Despite the presence of singularities of the temporal derivatives in typical solutions of \eqref{prob}, most papers dealing with the numerical analysis of finite difference methods for solving \eqref{prob} make the a priori assumption that higher-order temporal derivatives of the solution are smooth on the closed domain~$\bar Q$, in order to use Taylor expansions in their truncation error analyses. 
\end{remark}

\subsection{Purpose of paper}
The purpose of the present paper is to point out the severe consequence of assuming that the temporal derivative singularities described above are not present, i.e., the effect of assuming that $u_t\in C(\bar Q)$  when $0< \delta < 1$, or that $u_{tt}\in \bar Q$  when $1< \delta < 2$. Our main results (Theorem~\ref{th:main} and Corollary~\ref{cor:2}) will show that such solutions are a very restricted subclass of \eqref{prob}.

\section{Assuming more regularity in the solution}\label{sec:morereg}

\setcounter{section}{2}
\setcounter{equation}{0}\setcounter{theorem}{0}

In this section we shall assume more temporal regularity of $u$ than is present in Example~\ref{exa:1} and examine the effect of this arbitrary assumption.

The key to our analysis is the following basic result (see, e.g.,~\cite[Lemma 3.11]{Diet10}). We include its short elementary proof for completeness.

\begin{lemma}\label{lem:1}
Let $g\in C^\bard[0,T]$. Then 
\[
\lim_{t\to 0^+} D_t^{\delta}g(t) = 0.
\]
\end{lemma}
\begin{proof}
For any $t\in (0,T]$, 
\[
D_t^{\delta}g(t) =  \frac1{\Gamma(\bard-\delta)} \int_{s=0}^t (t-s)^{\bard-\delta-1} \,\frac{d^{\bard}g(s)}{ds^\bard}\, ds.
\]
But $g\in C^\bard[0,T]$ implies that $|d^\bard g(s)/ds^\bard|\le C$ for $0\le s\le T$ and some constant~$C$. Hence
\[
\left| D_t^{\bard-\delta}g(t) \right| \le  \frac{C}{\Gamma(\bard-\delta)} \int_{s=0}^t (t-s)^{\bard-\delta-1} \, ds
	= \frac{Ct^{\bard-\delta}}{\Gamma(\bard-\delta+1)}
	\to 0 \text{ as } t\to 0^+.
\]
\end{proof}

By considering the function $g(t) = t^{\delta}$, one sees easily that Lemma~\ref{lem:1} is no longer true if one assumes only that $g\in C^\bard(0,T]$.

The next result now follows immediately.

\begin{corollary}\label{cor:1}
In \eqref{prob}, assume that $\partial^\bard u(x,t)/\partial t^\bard$ is continuous on~$\bar Q$. Then 
\[
\lim_{t\to 0^+}D_t^\delta u(x,t)=0\quad\text{for each } x\in\Omega.
\]
\end{corollary}

Let $C^{2, \bard}(\bar Q)$ denote the space of functions $w\in C(\bar Q)$ for which $\partial w/\partial x_i$, $\partial^2 w/\partial x_i\partial x_j \in C(\bar Q)$ for $1\le i,j \le n$ and $\partial^k w/\partial t^k \in C(\bar Q)$ for $k=1,\bard$.

For functions $w\in C^2(\Omega)$, define the differential operator $L_0$ by
\[
L_0w(x) := -\sum_{i,j=1}^n p_{ij}(x,0)  \frac{\partial^2 w}{\partial x_i\partial x_j} + \sum_{i=1}^n q_i(x,0)  \frac{\partial w}{\partial x_i} + r(x,0) w\ \text{ for }x\in\Omega.
\]

We can now present our main result. In it the key assumption is the continuity of the temporal derivative $\partial^\bard u(x,t)/\partial t^\bard$ for $0\le t\le T$, not just for $0<t \le T$.

\begin{theorem}\label{th:main}
Suppose that the solution $u$ of \eqref{prob}  lies in $C^{2,\bard}(\bar Q)$. Then the initial value~$\phi_0(x)$ of $u$ must satisfy the equation $L_0\phi_0 = f$ on $\Omega$.
\end{theorem}
\begin{proof}
For each fixed $x\in\Omega$, consider the limit of   equation~\eqref{proba} as $t\to 0^+$.
Corollary~\ref{cor:1} gives  $\lim_{t\to 0^+}D_t^{\delta}u(x,t)=0$. Consequently we obtain $L_0u(x,0) = f(x,0)$, and this is true for each $x\in\Omega$.
\end{proof}

Theorem~\ref{th:main} shows that the assumption that the solution $u$ of \eqref{prob}  lies in $C^{2,\bard}(\bar Q)$ restricts the class of problems being studied because the initial condition $\phi_0$ cannot be chosen freely. The next example  illustrates the severity of this restriction.

\begin{example} \label{exa:2}
Consider the fractional heat equation   
\[
D_t^\delta v - \partial^2 v/\partial x^2 = 0\ \text{ for } (x,t)\in (0,\pi)\times(0,T]
\]
with $0<\delta<1$, boundary conditions $v(0,t) = v(\pi,t) =0$, and the initial condition $v(x,0) = \phi_0(x)$, where $\phi_0(x)$ is unspecified except that it satisfies the compatibility condition~\eqref{compatibility}. 

Suppose that we assume that the solution $v$ of this problem lies in $C^{2,1}(\bar Q)$. Then Theorem~\ref{th:main} and~\eqref{compatibility} show that $\phi_0$ must satisfy the conditions
\[
- \phi_0''(x) = 0 \ \text{ on } (0,\pi), \quad \phi_0(0) = \phi_0(\pi) = 0.
\]
These imply that $\phi_0 \equiv 0$. As all the data of this example are now zero, we get $v\equiv 0$. That is: imposing the arbitrary hypothesis that $v\in C^{2,1}(\bar Q)$ forces $v\equiv 0$.
\end{example}

\begin{remark}\label{rem:LinXuPaper}
The truncation error analysis in the widely-cited paper \cite{LX07} is carried out under the hypotheses of Example~\ref{exa:2}, to obtain the $(\Delta t)^2$ error term in \cite[(3.2)]{LX07}. This analysis is therefore valid only if the solution of the problem considered in~\cite[Section 3]{LX07} is identically zero. Nevertheless the \emph{stability} analysis of \cite[Section 3]{LX07} is unaffected by this observation,  and it can be combined with the truncation error analysis of~\cite{SORG16} to replace  $(\Delta t)^{2-\alpha}$ by $(\Delta t)^{\alpha}$ in the convergence results \cite[Theorem 3.2 (1) and Theorem 4.2 (1)]{LX07}---these new bounds hold true for functions whose derivatives behave as in Example~\ref{exa:1}. 
\end{remark}

In contrast to Example~\ref{exa:2}, Example~\ref{exa:1} exhibits a typical solution to the fractional heat equation: it lies in $C^{2,\bard}(Q)$ but not in $C^{2,\bard}(\bar Q)$.

One can generalise Example~\ref{exa:2} as follows. 

\begin{assumption}\label{Ass:1}
Assume that the boundary value problem
\begin{equation}\label{2ptbvp}
L_0w(x)=f(x,0) \ \text{ for }x\in\Omega, \quad  w(x,0)= \psi(x,0)\text{ for } x\in\partial\Omega
\end{equation}
has at most one solution $w(x)$.
\end{assumption}

Assumption~\ref{Ass:1} is satisfied if  $r(x,0)\ge 0$ for $x\in \Omega$, because then $L_0$ satisfies a maximum principle \cite[p.72]{PW84}. 

Alternatively, Assumption~\ref{Ass:1} is satisfied if, for example, $q\in C^1(\Omega)$ with
\[
p_{ij}(x,0)\equiv 1 \text{ for all }i,j \ \text{ and }\ r(x,0) - \frac12 \sum_{i=1}^n\frac{\partial q_i}{\partial x_i}(x,0) > 0\quad \text{for } x\in\Omega.
\]
For if $w_1$ and $w_2$ are two solutions of~\eqref{2ptbvp}, then  $L_0(w_1-w_2)=0$ on $\Omega$ and $w_1-w_2=0$ on $\partial\Omega$, so multiplying $L_0(w_1-w_2)=0$ by $w_1-w_2$ and integrating by parts over $\Omega$ yields 
\[
\int_\Omega \left[ |\nabla(w_1-w_2)|^2 + \left(r(x,0) - \frac12 \sum_{i=1}^n\frac{\partial q_i}{\partial x_i}(x,0)\right)(w_1-w_2)^2   \right]\,dx =0 
\]
which implies $w_1\equiv w_2$.

The next corollary prompted the title of our paper.

\begin{corollary}\label{cor:2}
Suppose that the solution $u$ of \eqref{prob}  lies in $C^{2,\bard}(\bar Q)$ and that Assumption~\ref{Ass:1} is satisfied. Then the initial value $\phi_0(x)$ of $u$ is determined uniquely by $L_0$ and $\psi$.
\end{corollary}
\begin{proof}
Theorem~\ref{th:main} shows that  $L_0\phi_0(x) = f(x,0)$ for all $x\in\Omega$. Furthermore, by~\eqref{compatibility} we have  $\phi_0(x)=\psi(x,0)$ for $x\in\partial\Omega$. The result now follows from Assumption~\ref{Ass:1}.
\end{proof}

It is clear that the conclusion of Corollary~\ref{cor:2} is unnatural---the initial condition \eqref{probc} should not be determined by the other data of the problem. This infelicity is caused by the unreasonable assumption that  $u\in C^{2,\bard}(\bar Q)$.

\begin{remark}
More general boundary conditions than those of \eqref{2ptbvp} can be considered by using \cite[p.70, Theorem 9]{PW84}.

When $n=1$ the classical spatial differentiation operator of \eqref{proba} can be replaced by a Caputo fractional derivative operator, using the associated maximum principle of \cite[Lemma 3.3]{AlR12a}.
\end{remark}

\section*{Acknowledgements}

The research of this author is supported in part by the National Natural Science Foundation of China under grant 91430216.


\end{document}